\documentclass[12pt]{amsart}

\usepackage[usenames,dvipsnames]{xcolor}

\usepackage{enumitem,kantlipsum}
\usepackage{comment}

\usepackage[utf8]{inputenc}
\usepackage{amssymb}
\usepackage{amsmath}
\usepackage{amsfonts}
\usepackage{amsthm}
\usepackage{tikz}

\usepackage{algpseudocode}

\usepackage{asymptote}
\usepackage{url}

\newtheorem{theorem}{Theorem}[section]
\newtheorem{corollary}[theorem]{Corollary}
\newtheorem{lemma}[theorem]{Lemma}
\newtheorem{proposition}[theorem]{Proposition}

\newtheorem{problem}{Problem}

\usepackage{graphicx}
\usepackage{caption}
\usepackage{subcaption}
\usepackage{times}

\theoremstyle{definition}
\newtheorem{alg}[theorem]{Algorithm}
\newtheorem{example}[theorem]{Example}

\newtheorem{defn}[theorem]{Definition}

\newcommand{\Z}{\mathbb{Z}}
\newcommand{\N}{\mathbb{N}}
\newcommand{\h}{\mathcal{H}}

\newcommand{\M}{\mathcal{M}}

\newtheorem{innercustomthm}{Theorem}
\newenvironment{customthm}[1]
  {\renewcommand\theinnercustomthm{#1}\innercustomthm}
  {\endinnercustomthm}

\newtheorem{innercustomprop}{Proposition}
\newenvironment{customprop}[1]
  {\renewcommand\theinnercustomprop{#1}\innercustomprop}
  {\endinnercustomprop}

\begin{document}
\title[Numerical Sets Associated to a Numerical Semigroup]{Enumerating numerical sets associated to a numerical semigroup}
\date{\today}

\author[Chen]{April Chen}
\address{Department of Mathematics, Harvard University, Cambridge, MA 02138}
\email{aprilchen@college.harvard.edu}

\author[Kaplan]{Nathan Kaplan}
\address{Department of Mathematics, University of California, Irvine, 419 Rowland Hall, Irvine, CA 92697}
\email{nckaplan@math.uci.edu}

\author[Lawson]{Liam Lawson}
\address{Department of Mathematics, University of California, Irvine, 419 Rowland Hall, Irvine, CA 92697}
\email{ljlawson@uci.edu}

\author[O'Neill]{Christopher O'Neill}
\address{Mathematics Department, San Diego State University, San Diego, CA 92182}
\email{cdoneill@sdsu.edu}

\author[Singhal]{Deepesh Singhal}
\address{Department of Mathematics, University of California, Irvine, 419 Rowland Hall, Irvine, CA 92697}
\email{singhald@uci.edu}

\begin{abstract}
A numerical set $T$ is a subset of $\N_0$ that contains $0$ and has finite complement.  The atom monoid of $T$ is the set of $x \in \N_0$ such that $x+T \subseteq T$.  Marzuola and Miller introduced the anti-atom problem:\ how many numerical sets have a given atom monoid?  This is equivalent to asking for the number of integer partitions with a given set of hook lengths.  We introduce the void poset of a numerical semigroup $S$ and show that numerical sets with atom monoid $S$ are in bijection with certain order ideals of this poset.  We use this characterization to answer the anti-atom problem when $S$ has small type.
\end{abstract}
\maketitle

\section{Introduction}


\subsection{The anti-atom problem}
A \emph{numerical set} $T$ is a subset of $\N_0 = \{0,1,2,\ldots\}$ that contains $0$ and has a finite complement. The elements of $\N_0\setminus T$ are called \emph{gaps}. The set of gaps of $T$ is denoted by $\h(T)$ and the number of gaps is the \emph{genus} of $T$, denoted by $g(T)$. The largest gap is the \emph{Frobenius number} of $T$, denoted by $F(T)$. A numerical set $S$ that is closed under addition is a \emph{numerical semigroup}.  We say that $n_1,\ldots, n_t \in S$ is a \emph{set of generators} of $S$ if $S$ is the set of all linear combinations of these elements with nonnegative integer coefficients.  That is,
\[
S = \langle n_1,\ldots, n_t \rangle = \{a_1 n_1 + \cdots + a_t n_t \mid a_1,\ldots, a_t \in \N_0\}.
\]
A numerical semigroup $S$ has a unique minimal set of generators and the cardinality of this set is the \emph{embedding dimension} of $S$, denoted by $e(S)$.  The smallest nonzero element of $S$ is called its \emph{multiplicity}, denoted by $m(S)$.

The set of \emph{pseudo-Frobenius} numbers of a numerical semigroup $S$ is defined by
\[
PF(S)=\{P\in\h(S)\mid P+S\setminus\{0\}\subseteq S\}.
\]
Clearly $F(S)$ is one of the pseudo-Frobenius numbers of $S$. The number of pseudo-Frobenius numbers of $S$ is the \emph{type} of $S$, denoted by $t(S)$.  

Antokoletz and Miller defined the \emph{atom monoid} of a numerical set $T$ in \cite{AdjointPaper} as
\[
A(T)=\{x\in \N_0 \mid x+T\subseteq T\}.
\]
It is easy to see that $A(T)$ is always a numerical semigroup contained in $T$. It is also referred to as the \emph{associated semigroup} of $T$. Marzuola and Miller raised the \emph{Anti-Atom problem} \cite{Miller}.
\begin{problem}[Anti-Atom problem]\label{anti_atom_problem}
Let $S$ be a numerical semigroup.  How many numerical sets $T$ have $A(T) = S$?
\end{problem}
\noindent This number is denoted by $P(S)$. The numerical sets $T$ for which $A(T)=S$ are called \emph{numerical sets associated to $S$}. In this paper we focus on computing $P(S)$.

Motivation for studying $P(S)$ comes from the theory of integer partitions.  A partition $\lambda$ has an associated multiset of hook lengths, denoted $\mathbb{H}(\lambda)$.  Let $H(\lambda)$ denote the underlying set of hook lengths of $\lambda$.  For detailed definitions and for a bijection between numerical sets and integer partitions, see \cite[Section 2]{NathanPartition}.  Keith and Nath show that $H(\lambda)$ is the set of gaps of a numerical semigroup $S$ and that numerical sets $T$ associated to $S$ are in bijection with partitions $\lambda$ with $H(\lambda) = \h(S)$ \cite[Corollary 1]{partition bijection}. (See also \cite[Proposition 3]{NathanPartition}.) Therefore, Problem \ref{anti_atom_problem} is equivalent to the following.
\begin{problem}\label{hookset}
Let $S$ be a numerical semigroup.  How many partitions $\lambda$ have $H(\lambda) = \h(S)$?
\end{problem}
\noindent This correspondence between partitions and numerical semigroups has been studied in several recent papers \cite{NathanPartition, Complementary numerical sets, Almost Sym Arf, decomposing partitions, Arf}.

\subsection{Previous results on $P(S)$}
It is easy to see that $T$ and $A(T)$ have the same Frobenius number. Since there are $2^{F-1}$ numerical sets with Frobenius number $F$, it follows that
\[
\sum_{S\colon F(S)=F}P(S)=2^{F-1}.
\]
Marzuola and Miller consider numerical semigroups of the form $N_F=\{0,F+1\rightarrow\}$, where the $\rightarrow$ indicates that all positive integers greater than $F+1$ are in $N_F$ \cite{Miller}. They prove that there is a positive constant $\gamma\approx 0.4844$ such that
\[
\lim_{F\to\infty}\frac{P(N_F)}{2^{F-1}}=\gamma.
\]
This means that nearly half of numerical sets $T$ with Frobenius number $F$ have $A(T) = N_F$.
In \cite{With Lin}, Singhal and Lin consider similar families of numerical semigroups.  Let $D$ be a finite set of positive integers, take $F>2\max(D)$ and define
\[
N(D,F)=\{0\}\cup\{F-l\mid l\in D\}\cup \{F+1\rightarrow\}.
\]
They show that for each $D$, there is a positive constant $\gamma_D$ such that
\[
\lim_{F\to\infty}\frac{P(N(D,F))}{2^{F-1}}=\gamma_D.
\]

Antokoletz and Miller define the \emph{dual} of a numerical set $T$ as
\[
T^*:=\{x\in\Z\mid F(T)-x\notin T\}
\]
and show that $A(T)=A(T^*)$ \cite{AdjointPaper}. Constantin, Houston-Edwards, and Kaplan interpret this construction in terms of partitions. They show that if $T$ corresponds to the partition $\lambda$, then $T^*$ corresponds to the conjugate partition $\Tilde{\lambda}$. Since $\mathbb{H}(\lambda) = \mathbb{H}(\Tilde{\lambda})$, we see that $H(A(T)) = H(A(T^*))$, which implies $A(T)=A(T^*)$ \cite[Proposition 12]{NathanPartition}.

\begin{theorem}\cite[Proposition 1]{Miller}\label{TBUS}
Let $S$ be a numerical semigroup. If $T$ is a numerical set with $A(T)=S$, then $S\subseteq T\subseteq S^*$.  Moreover, $A(S)=A(S^*)=S$.
\end{theorem}

Constantin, Houston-Edwards, and Kaplan define a \emph{missing pair} of $S$ to be a pair of gaps of $S$ that sums to $F(S)$ \cite[Section 7]{NathanPartition}. The set
\[
M(S) := \{a \colon a\not\in S,\ F(S) - a\not\in S\}
\]
of gaps in missing pairs is the \emph{void} of $S$.  They show that $S^*=S\cup M(S)$ \cite[Lemma 3]{NathanPartition}. It~is easy to check that $|M(S)|=2g(S)-F(S)-1$.

Theorem \ref{TBUS} implies that $P(S)=1$ if and only if $M(S)=\emptyset$.  A numerical semigroup $S$ for which $x \in S$ if and only if $F(S)-x \not\in S$ is called \emph{symmetric}. Therefore, $P(S) = 1$ if and only if $S$ is symmetric \cite[Corollary 2]{Miller}.  Fr\"oberg, Gottlieb, and H\"aggkvist prove that $S$ is symmetric if and only if $t(S) = 1$ \cite[Proposition 2]{Type}.

If $S$ is symmetric, then $F(S)$ is odd. There is a corresponding family of numerical semigroups with even Frobenius numbers. A numerical semigroup $S$ is called \emph{pseudo-symmetric} if $F(S)$ is even and $M(S)=\left\{\tfrac{1}{2} F(S)\right\}$. If $S$ is pseudo-symmetric, then $t(S) = 2$ \cite{Type}. Theorem \ref{TBUS} implies that if $S$ is pseudo-symmetric, then $P(S)=2$ \cite[Corollary 2]{Miller}.

These two families could lead one to guess a close relationship between the size of the void $|M(S)|$ and $P(S)$. However this relationship is subtle. A numerical semigroup with $P(S) = 2$ can have arbitrarily large $|M(S)|$ \cite[Proposition 16]{NathanPartition}. We do have an upper bound on $P(S)$ obtained from Theorem \ref{TBUS}.
\begin{corollary}\cite[Corollary 3]{NathanPartition}
For any numerical semigroup $S$,
\[
P(S)\leq 2^{|M(S)|}=2^{2g(S)-F(S)-1}.
\]
\end{corollary}

\subsection{The void poset and our main results}
We define a partial ordering on the void of a numerical semigroup $S$. Given $x,y\in M(S)$, we say $x\preccurlyeq y$ when $y-x\in S$. This poset is called the \emph{void poset} of $S$ and is denoted by $(\M(S),\preccurlyeq)$. In Proposition \ref{maxElementsArePF}, we show that the maximal elements of $(\M(S),\preccurlyeq)$ are precisely the pseudo-Frobenius numbers of $S$ other than $F(S)$.

It is known that every numerical semigroup $S$ satisfies $t(S) \le 2g(S)-F(S)$ \cite[Proposition 2.2]{H Nari}. If equality holds, then $S$ is called \emph{almost symmetric}. Both symmetric and pseudo-symmetric numerical semigroups are almost symmetric. In Proposition \ref{Almost Sym}, we show that almost symmetric numerical semigroups are those for which the void poset has no nontrivial relations.

A subset $I\subseteq M(S)$ is an \emph{order ideal} if for any $x,y\in M(S)$ satisfying $x\preccurlyeq y$ and $x\in I$, we have $y\in I$. If $A(T)=S$, then $T\setminus S\subseteq M(S)$. For an arbitrary subset $I\subseteq M(S)$, $I\cup S$ is not necessarily a numerical set associated to $S$. In Proposition \ref{posetlarger}, we show that the $I\subseteq M(S)$ for which $S \subseteq A(I\cup S)$ are precisely the order ideals of $(\M(S),\preccurlyeq)$.

In Section~\ref{sec:classify}, we introduce the notion of \emph{Frobenius triangles} of $S$ and use this concept to characterize the order ideals $I$ of $(\M(S),\preccurlyeq)$ for which $T = I \cup S$ is a numerical set associated to $S$ (Theorem \ref{Characterisation}).  This is one of the main results of this paper, as it yields an algorithm for computing $P(S)$ (Algorithm~\ref{a:frobtriangles}).  We also define the \emph{pseudo-Frobenius graph} of a numerical semigroup $S$ and use it to give a lower bound for $P(S)$ in Corollary \ref{lbound}.

We use these tools to analyze $P(S)$ for numerical semigroups of small type. As mentioned above, $t(S)=1$ if and only if $P(S)=1$. In Theorem \ref{t2p2} we show that if $t(S) = 2$, then $P(S) = 2$.  In Section \ref{sec:type3} we solve the much more difficult problem of characterizing $P(S)$ for numerical semigroups of type $3$.  In Theorem \ref{type 3 characterize} we prove that if $t(S) = 3$, then $P(S) \in \{2,3,4\}$, and we characterize in terms of the pseudo-Frobenius numbers of $S$ when each value occurs.

The relationship between $t(S)$ and $P(S)$ is not as straightforward for numerical semigroups of larger type. In Proposition \ref{t=4 P(S) large}, we show that a numerical semigroup of type $4$ can have $P(S)$ arbitrarily large. In Proposition \ref{t large P(S)=2} we show that given $t\geq 2$, there is a numerical semigroup with $t(S)=t$ and $P(S)=2$.  

A numerical semigroup $S$ is said to be of \emph{maximal embedding dimension} if $e(S)=m(S)$.  It is known that $t(S)\leq m(S)-1$ \cite[Corollary 1.23]{sanchesTextbook}, and that $S$ has maximal embedding dimension if and only if $t(S)=m(S)-1$ \cite[Corollary 2.2]{sanchesTextbook}.  Therefore, semigroups of maximal embedding dimension have a natural characterization in terms of their type.  In Section \ref{sec:large}, we compute $P(S)$ for a certain class of these semigroups.

\section{The void poset}

Recall that the void poset $(\M(S),\preccurlyeq)$ of $S$ is defined by $x\preccurlyeq y$ if and only if $y-x\in S$. For $x,y\in M(S)$, we write $x\prec y$ if $y-x\in S\setminus\{0\}$. In this section we use the structure of this poset to classify the numerical sets $T$ for which $S\subseteq A(T)$. In the next section we determine when such a numerical set also satisfies $A(T)\subseteq S$, thus classifying the numerical sets associated to $S$.

A poset $(\mathcal P,\preccurlyeq)$ is \emph{self-dual} if there exists a bijection $\phi:\mathcal P \rightarrow \mathcal P$ such that $a \preccurlyeq b$ if and only if $\phi(b) \preccurlyeq \phi(a)$.

\begin{lemma}\label{selfdual}
For any numerical semigroup $S$, $(\M(S),\preccurlyeq)$ is self-dual.
\end{lemma}
\begin{proof}
Define $\phi: M(S) \to M(S)$ by $\phi(x)=F(S)-x$. By the definition of $M(S),\ \phi$ is well-defined. The result follows by noting that $y-x=\phi(x)-\phi(y)$.
\end{proof}

We refer to this map as \emph{conjugation} on $(\M(S),\preccurlyeq)$. For $x \in M(S)$, define $\overline{x}=F(S)-x$.

\begin{proposition}\label{maxElementsArePF}
The set of maximal elements of $(\M(S),\preccurlyeq)$ is
\[
\max (\M(S),\preccurlyeq)=PF(S)\setminus\{F(S)\}.
\]
\end{proposition}
\begin{proof}
Let $P$ be a maximal element of $(\M(S),\preccurlyeq)$ and let $F = F(S)$. Given $s\in S\setminus\{0\}$, we know that $P+s\not\in M(S)$, since otherwise $P\prec P+s$. Thus, either $P+s\in S$ or $P+s\in\h(S)\setminus M(S)$. In the latter case, $F-P-s \in S$, but this implies $F-P=(F-P-s)+s \in S$, which contradicts the fact that $P \in M(S)$. We conclude that $P+s\in S$ and so $P\in PF(S)$. Moreover, since $P\in M(S)$ we know that $P\neq F$.

For the other direction, suppose $P\in PF(S)\setminus\{F\}$. If $F-P\in S$, then by the definition of $PF(S)$ we see that $F=P+(F-P)\in S$, which is a contradiction. Therefore, $F-P\not\in S$ and $P\in M(S)$. Next, if there is some $x\in M(S)$ such that $P\preccurlyeq x$, then $x-P\in S$ and $x\not\in S$. Since $P\in PF(S)$, the only way this could happen is if $x-P=0$, that is, $x=P$. We conclude that $P$ is a maximal element of $(\M(S),\preccurlyeq)$.
\end{proof}

Since $(\M(S),\preccurlyeq)$ is self-dual it follows that its minimal elements are
\[
\min (\M(S),\preccurlyeq)=\Big\{\overline{P}\mid P\in PF(S)\setminus\{F(S)\}\Big\}.
\]

\begin{example}
The void posets of $S_1 = \{0,4,8,10\rightarrow\}$ and $S_2=\langle 6, 25, 29 \rangle$ are as follows:
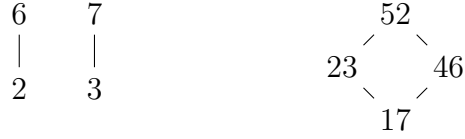
\begin{center}
\begin{tikzpicture}
\node (3) at (0,0) {$6$};
\node (4) [right of = 3] {$7$};
\node (2) [below of = 4] {$3$};
\node (1) [below of = 3] {$2$};
\draw [black] (1) -- (3);
\draw [black] (2) -- (4);

\node (24) at (5,0) {$52$};
\node (23) [below left of = 24] {$23$};
\node (22) [below right of = 24] {$46$};
\node (21) [below right of = 23] {$17$};
\draw [black] (21) -- (23);
\draw [black] (22) -- (24);
\draw [black] (21) -- (22);
\draw [black] (23) -- (24);
\end{tikzpicture}    
\end{center}
\end{example}

Recall that an order ideal of a poset is a subset $I$ such that if $x \in I$ and $x \preccurlyeq y$, then $y \in I$.
\begin{proposition}\label{posetlarger}
Let $S$ be a numerical semigroup with $M(S)=M$ and $I \subseteq M$. We have $S \subseteq A(I \cup S)$ if and only if $I$ is an order ideal of $(\M(S),\preccurlyeq)$.
\end{proposition}
\begin{proof}
First suppose that $I$ is an order ideal. We show that for each $s\in S$, $s+I \subseteq S \cup I$. Suppose $x \in I$.
\begin{itemize}
    \item Case 1: If $s + x \in S$, there is nothing else to check.
    \item Case 2: If $s + x \in \h(S)\setminus M(S)$, then $F-s-x \in S$. Therefore, 
    \[
    F-x= (F-s-x) + s \in S.
    \] 
    This contradicts the fact that $x \in M(S)$, so this case does not occur.
    \item Case 3: If $s + x \in M(S)$, then $ x \preccurlyeq s + x$ in $(\M(S),\preccurlyeq)$, so $s + x \in I$.
\end{itemize}

Conversely, suppose $S \subseteq A(I \cup S)$. Consider $x,y\in M(S)$ with $x\in I$ and $x\preccurlyeq y$. Then $y-x\in S$ and therefore  $y-x\in A(S\cup I)$.  That is,
\[
(y-x)+(S\cup I)\subseteq (S\cup I).
\]
In particular this implies that $y=(y-x)+x\in S\cup I$. Since $y\in M(S)$, we know that $y\notin S$ and therefore, $y\in I$. We conclude that $I$ is an order ideal of $(\M(S),\preccurlyeq)$.
\end{proof}

\begin{theorem} \label{t2p2}
Let $S$ be a numerical semigroup of type $2$. We have $P(S) = 2$.
\end{theorem}
\begin{proof}
Let $T$ be a numerical set associated to $S$. By Proposition \ref{posetlarger}, $T = S \cup I$ for some order ideal $I$ of $(\mathcal{M}(S),\preccurlyeq)$. We prove that either $I = M(S)$ and $T = S\cup M(S)$ or $I = \emptyset$ and $T = S$.

Suppose $PF(S) = \{P,F\}$ with $P<F$. We have $PF(S) \setminus \{F\}=\{P\}$, which means that $P$ is the unique maximal element of $(\M(S),\preccurlyeq)$. This also implies that $P=\max(M(S))$. Moreover, since $(\M(S),\preccurlyeq)$ is self-dual we also see that it has a unique minimal element $\overline{P} = F-P$.  We prove that if $I\neq\emptyset$, then $\overline{P} \in I$, which implies that $I = M(S)$.

Since $t(S)\neq 1$, we know that $P(S)\geq 2$. Suppose $T \neq S$. We know that $I=T\setminus S$ is a nonempty order ideal of $(\M(S),\preccurlyeq)$. Since $(\M(S),\preccurlyeq)$ has a unique maximal element, we must have $P\in I$. Since $P\not\in A(T)$, there is some $x\in T$ for which $P+x\not\in T$. Since $S=A(T)$, we know that $x\notin S$. This means that $x\in I\subseteq M(S)$. We have seen that $F-P$ is the unique minimal element of $(\M(S),\preccurlyeq)$, so $F-P\leq x$.

Now if $F-P<x$, then $F<P+x$, which would contradict the fact that $P+x\notin T$. Therefore $F-P=x$ and hence $F-P\in I$.
Finally, since the unique minimal element of $(\M(S),\preccurlyeq)$ is in $I$, we conclude that $I = M(S)$.
\end{proof}

\begin{example}
Consider $S=\langle19,21,24\rangle$, which has $t(S) = 2$ and $PF(S)=\{98,113\}$. Note that $(\M(S),\preccurlyeq)$ has a unique maximal element and a unique minimal element. If $I\cup S$ is a numerical set associated to $S$, then either $I=\emptyset$ or $I=M(S)$.
\begin{figure}[h]
    \centering
    \begin{tikzpicture}
    \node (1) at (0,0) {$98$};
    \node (2) at (-1,-0.7) {$74$};
    \node (3) at (0,-0.7) {$77$};
    \node (4) at (1,-0.7) {$79$};
    \node (5) at (-1.5,-1.4) {$53$};
    \node (6) at (-0.5,-1.4) {$55$};
    \node (7) at (0.5,-1.4) {$58$};
    \node (8) at (1.5,-1.4) {$60$};
    \node (9) at (-1,-2.1) {$34$};
    \node (10) at (0,-2.1) {$36$};
    \node (11) at (1,-2.1) {$39$};
    \node (12) at (0,-2.8) {$15$};
    \draw [black] (1) -- (2);
    \draw [black] (1) -- (3);
    \draw [black] (1) -- (4);
    \draw [black] (2) -- (5);
    \draw [black] (2) -- (6);
    \draw [black] (3) -- (5);
    \draw [black] (3) -- (7);
    \draw [black] (4) -- (6);
    \draw [black] (4) -- (7);
    \draw [black] (4) -- (8);
    \draw [black] (9) -- (5);
    \draw [black] (9) -- (6);
    \draw [black] (9) -- (7);
    \draw [black] (10) -- (6);
    \draw [black] (10) -- (8);
    \draw [black] (11) -- (7);
    \draw [black] (11) -- (8);
    \draw [black] (12) -- (9);
    \draw [black] (12) -- (10);
    \draw [black] (12) -- (11);
    \end{tikzpicture}
    \caption{Void poset of $S=\langle19,21,24\rangle$}
    \label{fig:my_label}
\end{figure}
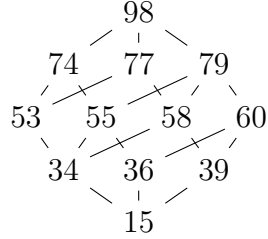
\end{example}

\begin{proposition}
Let $S$ be a numerical semigroup.
\begin{enumerate}
\item We have $M(S)=\emptyset$ if and only if $S$ is symmetric.
\item We have $|M(S)|=1$ if and only if $S$ is pseudo-symmetric.
\end{enumerate}
\end{proposition}
\begin{proof}
Recall that $|M(S)|=2g(S)-F(S)-1$. A numerical semigroup is symmetric if and only if $F(S)=2g(S)-1$ \cite[Lemma 1]{Type}. This occurs if and only if $M(S)=\emptyset$. A numerical semigroup is pseudo-symmetric if and only if $F(S)=2g(S)-2$ \cite[Lemma 3]{Type}. This happens if and only if $|M(S)|=1$.
\end{proof}

\begin{proposition}\cite[Proposition 2.2]{H Nari}\label{type_bound}
We have
\[
t(S)\leq 2g(S)-F(S) = |M(S)| + 1.
\]
\end{proposition}
\begin{proof}
We know that $PF(S)\setminus\{F\}\subseteq M(S)$. Therefore, $t(S)-1\leq 2g(S)-F(S)-1$.
\end{proof}

Recall that $S$ is almost symmetric if and only if $t(S)=2g(S)-F(S)$.

\begin{proposition}\label{Almost Sym}
A numerical semigroup $S$ is \emph{almost symmetric} if and only if $(\M(S),\preccurlyeq)$ has no nontrivial relations, that is, $x \preccurlyeq y$ implies $x = y$.
\end{proposition}
\begin{proof}
Since $|PF(S)\setminus\{F(S)\}| = t(S) -1$ and $PF(S) \setminus \{F(S)\} \subseteq M(S)$, we see that $S$ is almost symmetric if and only if $PF(S)\setminus\{F\}= M(S)$. Proposition \ref{maxElementsArePF} says that $PF(S)\setminus\{F(S)\}$ is the set of maximal elements of $(\M(S),\preccurlyeq)$. Thus, $S$ is almost symmetric if and only if all elements of $(\M(S),\preccurlyeq)$ are maximal. This is equivalent to $(\M(S),\preccurlyeq)$ having no nontrivial relations.
\end{proof}

We end this section with a quick observation about the parity of $P(S)$.  Theorem \ref{TBUS} and the concept of the dual of a numerical set leads directly to the following result.
\begin{proposition}
If $F(S)$ is even, then $P(S)$ is even.
\end{proposition}
\begin{proof}
Given a numerical set $T$ with $F(T)=F(S)$, note that $\frac{F(S)}{2}\in T$ if and only if $\frac{F(S)}{2}\not\in T^*$. This means $T\neq T^*$. Therefore, we can divide the numerical sets associated to $S$ into pairs.
\end{proof}

\section{Classifying associated numerical sets}\label{sec:classify}

We have seen that if $T$ is a numerical set associated to $S$, it is necessary that $T\setminus S\subseteq M(S)$ and that $T\setminus S$ is an order ideal of $(\M(S),\preccurlyeq)$. In this section we study a class of order ideals that always give rise to numerical sets associated to $S$. This gives a lower bound for $P(S)$. We also obtain a complete characterization of the order ideals that lead to numerical sets associated to $S$.

\begin{defn}
We say a triple $(P,x,y) \in PF(S) \times M(S)^2$ satisfying $P + x + y = F(S)$ is a \emph{Frobenius triangle} of $S$, and let
\[
Tr(S)=\{(P,x,y)\mid P\in PF(S),\; x,y\in M(S),\ P+x+y=F(S)\}
\]
denote the set of Frobenius triangles of $S$.  
Given an order ideal $I\subseteq M(S)$ and a Frobenius triangle $(P,x,y)$, we say that $I$ \emph{satisfies the Frobenius triangle} $(P,x,y)$ if $P$, $x \in I$ and $F(S) - y \notin I$. 
The \emph{pseudo-Frobenius graph} of $S$, denoted by $GPF(S)$, is the graph with vertices $PF(S)\setminus\{F(S)\}$ that has an edge between $P,Q$ if and only if $P+Q-F(S)\in S$.
We denote the number of connected components of $GPF(S)$ by $\kappa(S)$.
\end{defn}



An order ideal $I$ of $(\M(S),\preccurlyeq)$ is self-dual if $x\in I$ implies that $\overline{x}\in I$.

\begin{proposition}\label{sdideals}
If $S$ is a numerical semigroup and $I$ is a self-dual order ideal of $(\M(S),\preccurlyeq~)$, then $I\cup S$ is a numerical set associated to $S$.
\end{proposition}
\begin{proof}
By Proposition \ref{posetlarger}, $S\subseteq A(I\cup S)$. Since $I$ is self-dual, $x \in I$ implies $F(S)-x \in I$. However, 
\[
x+(F(S)-x)=F(S) \not\in I\cup S.
\] 
So $x+(I\cup S)\not\subseteq (I\cup S)$ and hence $x\not\in A(I\cup S)$. We conclude that $A(I\cup S)=S$.
\end{proof}

Since self-dual order ideals give numerical sets associated to $S$, the number of self-dual order ideals of $(\M(S),\preccurlyeq)$ is a lower bound for $P(S)$.  We now give a simpler description of self-dual order ideals.
\begin{lemma}\label{check PF self-dual}
Let $I$ be an order ideal of $(\M(S),\preccurlyeq)$. If for each $P\in PF(S)\cap I$ we have $\overline{P}\in I$, then $I$ is self-dual.
\end{lemma}
\begin{proof}
Consider $x\in I$. Pick a maximal element above $x$, say $P\in PF(S) \setminus\{F(S)\}$ satisfies $x\preccurlyeq P$. Since $I$ is an order ideal, we have $P\in I$. We are given that $\overline{P}\in I$.  Since $x\preccurlyeq P$, we see that $\overline{P}\preccurlyeq \overline{x}$.  This implies $\overline{x}\in I$. We conclude that $I$ is self-dual.
\end{proof}

\begin{lemma}\label{dual}
Let $I$ be a self-dual order ideal of $(\M(S),\preccurlyeq)$. If $x\in I,\ y\in M(S)$, and $y\preccurlyeq x$, then $y\in I$.
\end{lemma}
\begin{proof}
Since $I$ is self-dual, we have $\overline{x}\in I$. Next, since $\overline{x}\preccurlyeq\overline{y}$ and $I$ is an order ideal, we have $\overline{y}\in I$. Finally, since $I$ is self-dual, we conclude that $y\in I$.
\end{proof}

\begin{lemma}\label{determined}
If $I_1$, $I_2$ are self-dual order ideals of $(\M(S),\preccurlyeq)$ and
\[
I_1\cap PF(S)=I_2\cap PF(S),
\]
then $I_1=I_2$.
\end{lemma}
\begin{proof}
Given $x\in I_1$, pick a maximal element $P$ above it, say $P\in PF(S)\setminus \{F(S)\}$ satisfies $x\preccurlyeq P$. Since $I_1$ is an order ideal, $P\in I_1$.  Therefore $P$ is in $I_1\cap PF(S)$ and so $P\in I_2$ also. By Lemma \ref{dual}, $x\in I_2$. This shows that $I_1\subseteq I_2$. By symmetry $I_1=I_2$.
\end{proof}

Recall that the pseudo-Frobenius graph of $S$ has vertex set $PF(S)\setminus\{F(S)\}$ and has an edge between $P$ and $Q$ when $P+Q-F(S)\in S$.
Note that there is an edge between $P$ and $Q$ if and only if $\overline{P}\preccurlyeq Q$.  This is  equivalent to $\overline{Q} \preccurlyeq P$.  Also note that $GPF(S)$ may possibly have loops. See Example \ref{eg 05->}.

\begin{theorem}\label{self-dual count}
Let $S$ be a numerical semigroup.  If $I$ is a self-dual order ideal of $(\M(S),\preccurlyeq)$ then $I\cap PF(S)$ is a union of connected components of $GPF(S)$.

Conversely, for any union of connected components of $GPF(S)$ there is a unique self-dual order ideal that contains precisely this set of pseudo-Frobenius numbers.
\end{theorem}
\begin{proof}
Suppose $C_1,\ldots, C_{\kappa}$ are the connected components of $GPF(S)$. Let $I$ be a self-dual order ideal of $(\M(S),\preccurlyeq)$. Suppose $P \in I \cap PF(S)$ and $P,Q \in C_i$ for some $i$. There is a path $P=P_0,P_1,P_2,\dots,P_{n-1},P_n=Q$ in $GPF(S)$ from $P$ to $Q$.  If $P_i\in I$, then since $I$ is self-dual, we have $\overline{P_i}\in I$. Moreover, we know that $\overline{P_i}\preccurlyeq P_{i+1}$ since there is an edge between $P_i$ and $P_{i+1}$ in $GPF(S)$. This implies $P_{i+1}\in I$. By induction, we conclude that $Q\in I$.
We have shown that if $I\cap C_i\neq\emptyset$, then $C_i \subseteq I$. This means that $I\cap PF(S)$ is a union of connected components of $GPF(S)$.

Conversely, let $J$ be a subset of $\{1,2,\dots,\kappa\}$. Let
\[
I=\{x\in M(S) \mid \exists\ i\in J \text{ and }\ P\in C_i \text{ with } x\preccurlyeq P\}.
\]
It is clear that $I\cap PF(S)= \bigcup_{i \in J}C_i$.  We now show that $I$ is an order ideal. Suppose $a\in I$, $b\in M(S)$ and $a\preccurlyeq b$. Since $a\in I$, we know that $\exists\ i\in J$, $P\in C_i$ such that $a\preccurlyeq P$. Also consider a minimal element below $a$. Such an element is of the form $\overline{Q}$ where $Q\in PF(S)$. Since $\overline{Q}\preccurlyeq a \preccurlyeq P$, there is an edge between $Q$ and $P$ in $GPF(S)$. In particular, $Q\in C_i$. Next, consider a maximal element above $b$, say $R\in PF(S)\setminus \{F(S)\}$ satisfies $b\preccurlyeq R$. Note that $\overline{Q}\preccurlyeq R$, which implies that there is an edge between $Q$ and $R$ in $GPF(S)$.  Therefore, $R\in C_i$, and since $b\preccurlyeq R$ we conclude that $b\in I$ and $I$ is an order ideal.

We now show that $I$ is self-dual. Let $a,P,Q$ be as above. Since $\overline{Q}\preccurlyeq a\preccurlyeq P$ we have $\overline{P}\preccurlyeq \overline{a}\preccurlyeq Q$. Since $Q\in C_i$ it follows that $\overline{a}\in I$. We conclude that $I$ is a self-dual order ideal with 
\[
I\cap PF(S)= \bigcup_{i \in J}C_i.
\]
Lemma \ref{determined} implies that $I$ is the unique self-dual order ideal with this set of pseudo-Frobenius numbers.
\end{proof}

\begin{corollary}\label{lbound}
We have $P(S)\geq 2^{\kappa(S)}$.
\end{corollary}

\begin{example}\label{eg 05->}
Consider $S=\{0,5\rightarrow\}$, so $F(S)=4$ and $M(S)=\{1,2,3\}$. The poset $(\M(S),\preccurlyeq)$ has no nontrivial relations. Therefore $PF(S)=\{1,2,3,4\}$ and $S$ is almost symmetric. See Figure \ref{fig 05} for the poset $(\M(S),\preccurlyeq)$ and the graph $GPF(S)$. The graph $GPF(S)$ has $2$ connected components. The self-dual order ideals of $(\M(S),\preccurlyeq)$ give $4$ numerical sets associated to $S$: $T_1=S$, $T_2=\{1,3\}\cup S$, $T_3=\{2\}\cup S$ and $T_4=\{1,2,3\}\cup S=S^*$. There are two more numerical sets associated to $S$, which come from order ideals that are not self-dual: $T_5=\{1\}\cup S$, $T_6=\{1,2\}\cup S$. We have $P(S)=6$.

\begin{figure}[!tbp]
  \begin{subfigure}[b]{0.3\textwidth}
    \begin{tikzpicture}
    \node (1) at (0,0) {$1$};
    \node (2) at (1,0) {$2$};
    \node (3) at (2,0) {$3$};
    \end{tikzpicture}
    \caption{Void poset}
    \label{fig:f1}
  \end{subfigure}
  \begin{subfigure}[b]{0.3\textwidth}
    \begin{tikzpicture}[main/.style = {draw, circle}] 
    \node[main] (1) at (-1,0) {$1$};
    \node[main] (2) at (0,0) {$3$};
    \node[main] (3) at (1,0) {$2$};
    \draw (3) to [out=45,in=135,looseness=4] (3);
    \draw (1) -- (2);
    \end{tikzpicture}
    \caption{$GPF$ graph}
    \label{fig:f2}
  \end{subfigure}
  \caption{Void poset and $GPF$ graph of $\{0,5\rightarrow\}$}
  \label{fig 05}
\end{figure}
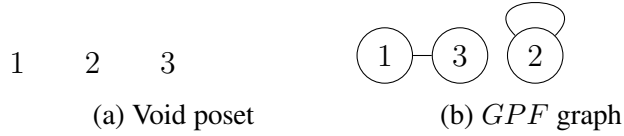

\end{example}

Next we give a complete characterization of the order ideals that lead to numerical sets associated to $S$. Given $P\in PF(S)\setminus\{F(S)\}$, define
\[
Tr_P(S)=\{(P,x,y) \mid x,y\in M(S) \text{ with } P+x+y=F(S)\} \subset Tr(S).
\]
Note that given $P\in PF(S) \setminus\{F(S)\}$ and $x,y\in M(S)$,
$(P,x,y)$ is a Frobenius triangle if and only if $P+x=\overline{y}$.
Recall that given an order ideal $I$ and a Frobenius triangle $(P,x,y)$, $I$ satisfies the Frobenius triangle $(P,x,y)$ if $P,x\in I$ and $\overline{y}\notin I$.

\begin{theorem}\label{Characterisation}
Let $S$ be a numerical semigroup, $I \subseteq M(S)$, and $T=I\cup S$. Then $T$ is a numerical set associated to $S$ if and only if
\begin{enumerate}
\item $I$ is an order ideal of $(\M(S),\preccurlyeq)$, and 
\item for each $P \in I \cap PF(S)$, one of the following conditions is satisfied:\ $2P \not\in S$, $F(S)-P \in I$, or there is a Frobenius triangle $(P,x,y)$ that is satisfied by $I$.
\end{enumerate}
Moreover, if $I\cup S$ is a numerical set associated to $S$, then for all $P\in I\cap PF(S)$ 
either $\overline{P} \in I$ or there is a Frobenius triangle $(P,x,y)$ that is satisfied by~$I$.
\end{theorem}
\begin{proof}
If $T$ is a numerical set associated to $S$, then Proposition \ref{posetlarger} implies that $I$ is an order ideal of $(\M(S),\preccurlyeq)$.  Suppose $P\in I\cap PF(S)$. We know that $P\not\in A(T)=S$, that is, $P+T\not\subseteq T$. This means that there is some $x\in T$ for which $P+x\not\in T$. Since $S=A(T)$, we know that $x\notin S$ and hence $x\in I$.
\begin{itemize}[wide, labelindent=0pt] 
    \item Case 1: Suppose $P+x\in M(S)\setminus T$. Then $\left(P,x,\overline{P+x}\right)$ is a Frobenius triangle and $I$ satisfies it.
    \item Case 2: Suppose $P+x\in \h(S)\setminus M(S)$. This means that $F-P-x\in S$. Now suppose $F-P-x\neq 0$. Then since $P\in PF(S)$, we see that
    \[
    F-x=P+(F-P-x)\in S.
    \]
    This contradicts the fact that $x\in I\subseteq M(S)$. Therefore $F-P-x=0$, that is, $\overline{P}=x\in I$.
\end{itemize}

Next we prove the other direction. Suppose that $I$ satisfies (1) and (2) in the theorem statement. By Proposition \ref{posetlarger}, $S\subseteq A(T)$. Assume for the sake of contradiction that $S\subsetneq A(T)$. Let $P=\max(A(T)\setminus S)$. Since $F(S) \notin T$, we know that $P\neq F(S)$. Given $s\in S\setminus\{0\}$, we know that $P+s\in A(T)$ as $A(T)$ is closed under addition. Since $P=\max(A(T)\setminus S)$ it follows that $P+s\in S$. This means that $P\in I\cap PF(S)$. Similarly $2P\in A(T)$, and since $2P>P$, we conclude that $2P\in S$. By (2) we know that either $\overline{P}\in I$ or there is some Frobenius triangle that $I$ satisfies.
\begin{itemize}[wide, labelindent=0pt] 
    \item Case 1: Suppose $\overline{P}\in I$. Then $P+\overline{P}=F\not\in T$. This contradicts the fact that $P\in A(T)$.
    \item Case 2: Suppose that $I$ satisfies the Frobenius triangle $(P,x,y)$. This means that $x\in I$ and $P+x=\overline{y}\notin I$. This again contradicts the fact that $P\in A(T)$.
\end{itemize}
We get a contradiction in both cases.  Therefore, $S=A(T)$, that is, $T$ is a numerical set associated to $S$.
\end{proof}

A numerical semigroup $S$ is called \emph{P-minimal} if $P(S)=2^{\kappa(S)}$.
\begin{corollary}\label{P(S) no triangle}
If $S$ is a numerical semigroup for which $Tr(S)=\emptyset$, then $S$ is P-minimal.
\end{corollary}

Note that the converse of Corollary \ref{P(S) no triangle} is not true. For example, for $S=\langle 8, 12, 13, 23, 30\rangle$ we have $\kappa(S)=1$ and $P(S)=2$, so $S$ is P-minimal. However, $(17,5,5)\in Tr(S)$, so $Tr(S)\neq\emptyset$.

\section{Structure among Frobenius triangles}

This section consists of some technical lemmas concerning Frobenius triangles. These will be useful in computing $P(S)$ for certain classes of numerical semigroups. If $P,Q\in PF(S)\setminus\{F(S)\}$ and $P-Q\in M(S)$, then $(Q,P-Q,F(S)-P)$ is a Frobenius triangle in $Tr_Q(S)$. Our main result in this section is the following.
\begin{proposition}\label{condition to Frobenius triangle}
Let $Q\in PF(S)\setminus\{F(S)\}$. We have $Tr_Q(S)\neq\emptyset$ if and only if $\exists\ P\in PF(S)\setminus\{F(S)\}$ such that $P-Q\in M(S)$.
\end{proposition}

Before proving this statement we need some preliminary results.
\begin{lemma}\label{largest PF}
Let $P=\max_{<}(M(S))$. We have $P\in PF(S)\setminus\{F(S)\}$ and $Tr_P(S)=\emptyset$.
\end{lemma}
\begin{proof}
If $P=\max_{<}(M(S))$, then $P$ is a maximal element of $(\M(S),\preccurlyeq)$. Therefore, $P\in PF(S)\setminus\{F(S)\}$. Moreover, $\min_{<}(M(S))=F(S)-P$, so $Tr_P(S)=\emptyset$.
\end{proof}

\begin{lemma}\label{triless}
Suppose $Q\in PF(S)\setminus\{F(S)\}$ and $(Q,x,y)\in Tr_Q(S)$. 
\begin{enumerate}
\item If $a \in M(S)$ satisfies $a\prec x$, then $a\preccurlyeq \overline{y}$.
\item If $b \in M(S)$ satisfies $\overline{y}\prec b$, then $x\preccurlyeq b$.
\end{enumerate}
\end{lemma}
\begin{proof}
We know that $Q+x+y=F(S)$, that is, $\overline{y}-x=Q$. If $x-a \in S \setminus \{0\}$, we have
\[
\overline{y}-a=(\overline{y}-x)+(x-a)=Q+(x-a)\in S,
\]
since $Q\in PF(S)$.  If $b-\overline{y} \in S \setminus \{0\}$, we have
\[
b-x=(b-\overline{y})+(\overline{y}-x)=Q+(b-\overline{y})\in S.
\]
\end{proof}

\begin{corollary}\label{PFtriless}
If $P,Q\in PF(S)\setminus\{F(S)\}$, $a\in M(S)$, and $P-Q\in M(S)$, then $a\prec P-Q$ implies $a\prec P$.
\end{corollary}
\begin{proof}
Note that $(Q,P-Q, F(S)-P)\in Tr_Q(S)$.
\end{proof}

\begin{lemma}\label{l:abovebelowtriangle}
Suppose $I$ is an order ideal of $(\M(S),\preccurlyeq)$ and $(Q,x,y)\in Tr_Q(S)$ is satisfied by $I$.  Then $x$ is a minimal element of $I$ and $\overline{y}$ is a maximal element of $M(S) \setminus I$.
\end{lemma}
\begin{proof}
By assumption, $Q,x\in I$ and $\overline{y}\notin I$. If $a \in M(S)$ satisfies $a \prec x$, then Lemma \ref{triless} implies that $a\preccurlyeq \overline{y}$. Therefore, $a \not\in I$.  If $b \in M(S)$ satisfies $\overline{y} \prec b$, then Lemma \ref{triless} implies that $x\preccurlyeq b$.  Therefore, $b \in I$.
\end{proof}

\begin{lemma}\label{superlemma}
Suppose $Q\in PF(S)\setminus\{F(S)\}$ and $(Q,x,y)\in Tr_Q(S)$. If $\overline{P}\preccurlyeq y$ for some $P\in PF(S)\setminus\{F(S)\}$, then $P-Q\in M(S)$ and $x\preccurlyeq P-Q$.
\end{lemma}
\begin{proof}
We want to show that $P-Q\in M(S)$. If this is not the case then either $P-Q\in S$ or $F(S)-(P-Q)\in S$.
\begin{itemize}[wide, labelindent=0pt] 
    \item Suppose $P-Q\in S$. Then $P\preccurlyeq Q$ in $(\M(S),\preccurlyeq)$. Since $P$ is a maximal element of $(\M(S),\preccurlyeq)$ we see that $P=Q$. Since $\overline{P}\preccurlyeq y$ we have $y-(F(S)-P)\in S$. But $y+P-F(S)=x$, which contradicts the fact that $x\in M(S)$.
    \item Suppose $F(S)-(P-Q)\in S$. Since $y-\overline{P}=y-F(S)+P\in S$, we have
    \[
    (F(S)-P+Q)+(y-F(S)+P)= Q+y\in S.
    \]
    But $Q+y=F(S)-x$, which contradicts the fact that $x\in M(S)$.
\end{itemize}
We conclude that $P-Q\in M(S)$. Finally,
\[
(P-Q)-x=P-(Q+x)=P-(F(S)-y)=y-\overline{P}\in S,
\]
and so $x\preccurlyeq P-Q$.
\end{proof}

\begin{proof}[Proof of Proposition \ref{condition to Frobenius triangle}]
If $(Q,x,y)\in Tr_Q(S)$ then $y$ would be above some minimal element of $(\M(S),\preccurlyeq)$, say $P\in PF(S) \setminus \{F(S)\}$ satisfies $\overline{P}\preccurlyeq y$. Lemma \ref{superlemma} implies that $P-Q\in M(S)$.

Conversely, if there is a $P\in PF(S)\setminus\{F(S)\}$ such that $P-Q\in M(S)$, then $(Q,P-Q,F(S)-P)\in Tr_Q(S)$.
\end{proof}

\begin{defn}
A numerical semigroup $S$ is \emph{triangle-free} if whenever $P_1,P_2 \in PF(S)$ satisfy $P_1-P_2 \in M(S)$, then $P_1=F(S)$.
\end{defn}

Proposition \ref{condition to Frobenius triangle} implies that $Tr(S)=\emptyset$ if and only if $S$ is triangle-free.

\begin{proposition}\label{none P-Q}
If $S$ is triangle-free, then it is P-minimal.
\end{proposition}
\begin{proof}
If $S$ is triangle-free then by Proposition \ref{condition to Frobenius triangle}, $Tr(S)=\emptyset$. By Corollary \ref{P(S) no triangle}, $S$ is P-minimal.
\end{proof}

\section{Algorithm to determine $P(S)$}

In this section, we present an algorithm to compute $P(S)$ given a numerical semigroup~$S$.  The algorithm essentially works by computing, for each Frobenius triangle $(P,x,y)$, the list of order ideals that satisfy $(P,x,y)$  (Theorem~\ref{Characterisation}), as well as using Lemma~\ref{l:abovebelowtriangle} to further restrict the list of elements considered.  

\begin{alg}\label{a:frobtriangles}
Computes $P(S)$ for a numerical semigroup $S$ using Frobenius triangles.
\begin{algorithmic}
\Function{AssociatedNumericalSets}{$S$}
\State $C_P \gets \{(P, F(S) - P, 0)\} \cup Tr_P(S)$ for each $P \in PF(S) \setminus \{F(S)\}$
\State $R \gets \emptyset$
\ForAll{$A \subseteq PF(S) \setminus \{F(S)\}$ and $a \in \prod_{P\in A} C_P$}
    \State $G \gets \{z \in M(S) \mid z - x \in S$ for some $a_P = (P,x,y)\}$
    \State $B_1 \gets \{z \in M(S) \mid F(S) - y - z \in S$ for some $a_P = (P,x,y)\}$
    \State $B_2 \gets \{z \in M(S) \mid P-z \in S$ for some $P \in PF(S) \setminus A\}$
    \If{$G \cap B_1 = \emptyset$ and $G \cap B_2 = \emptyset$}
        \State $Z \gets M(S) \setminus (G \cup B_1 \cup B_2)$
        \State Add to $R$ the numerical set $A \cup G \cup Y \cup S$ for each order ideal $Y$ of $(Z, \preccurlyeq)$
    \EndIf
\EndFor
\State \Return $R$
\EndFunction
\end{algorithmic}
\end{alg}

Algorithm \ref{a:frobtriangles} works as follows.  Each numerical set $T$ associated to $S$ is enumerated by first choosing the set $A = PF(S) \cap T$ of pseudo-Frobenius numbers in $T$, followed by choosing, for each $P \in A$, either a Frobenius triangle $(P, x, y)$ satisfied by the order ideal $I = T \setminus S$ or having $\overline{P}\in I$ as in Theorem~\ref{Characterisation}.  Since $I$ is an order ideal, every element of the set $G$ must lie in $I$.  Moreover, every element of $B_1$ must lie outside of $I$ by Lemma~\ref{l:abovebelowtriangle}, and any element of $B_2$ must lie outside of $I$ since $A$ contains all maximal elements of $I$.  By Theorem~\ref{Characterisation}, these are the only conditions on $T$, so any such order ideal $I$ of $(M(S),\preccurlyeq)$ containing $G$ and avoiding $B_1$ and $B_2$ yields a numerical set $I \cup S$.  

We may obtain a slight improvement on the main loop in Algorithm~\ref{a:frobtriangles} by using a recursive implementation.  Rather than trying all possible collections $a$ of Frobenius triangles, recursively add Frobenius numbers to $a$ one at a time, each time growing $G$, $B_1$, and $B_2$ appropriately.  This allows one to break out whenever a newly added Frobenius triangle renders $G \cap B_1$ or $G \cap B_2$ nonempty, thereby saving iterations of the main loop.  Table~\ref{tb:algruntimes} contains sample runtimes for a \texttt{Sage} implementation of Algorithm~\ref{a:frobtriangles}, both with and without utilizing this recursive enhancement.  

\begin{table}
\begin{tabular}{l|l|l|l|l|l}
$S$ & $\mathsf e(S)$ & $\mathsf t(S)$ & $P(S)$ & Algorithm~\ref{a:frobtriangles} & Recursive \\
\hline
$\langle 271, 309, 352, 422 \rangle$         & 4  & 4  & 2       & 0.63 s & 0.61 s   \\
$\langle 871, 909, 952, 1022 \rangle$        & 4  & 4  & 2       & 2.3 s & 2.3 s  \\
$\langle 603, 608, 613, \ldots, 653 \rangle$ & 11 & 2  & 2       & 2.1 s & 2.1 s  \\
$\langle 49, 342, 349, 350 \rangle$          & 4  & 13 & 2       & 725.0 s & 0.4 s  \\
$\langle 10, 101, 102, \ldots, 109 \rangle$  & 10 & 9  & 126905  & 60 s & 14.7 s  \\
\end{tabular}
\medskip
\caption{Runtimes for $P(S)$ computations, each using \texttt{GAP} and the package \texttt{numericalsgps}~\cite{numericalsgpsgap}.}
\label{tb:algruntimes}
\end{table}

\section{Numerical semigroups of type 3}\label{sec:type3}

Recall that a numerical semigroup $S$ has type $1$ if and only if $P(S) = 1$.  In Theorem~$\ref{t2p2}$, we proved that $t(S)=2$ implies $P(S) = 2$. In this section we show that $t(S) = 3$ implies $P(S)\in \{2,3,4\}$ and give conditions on $S$ that characterize when each possibility occurs.  Throughout this section, $S$ is a numerical semigroup of type $3$ with pseudo-Frobenius numbers $P<Q<F$.

\begin{lemma}
If $S$ is a numerical semigroup with $t(S)=3$ and $GPF(S)$ has two connected components, then $P(S)=4$.
\end{lemma}

\begin{proof}
By Theorem \ref{self-dual count} we know that there are $4$ self-dual order ideals of $(\M(S),\preccurlyeq)$, which lead to $4$ numerical sets associated to $S$. Assume for the sake of contradiction that there is an order ideal $I$ of $(\M(S),\preccurlyeq)$, that is not self-dual, for which $I\cup S$ is a numerical set associated to $S$.  If $GPF(S)$ has two connected components, then $P+Q-F \not\in S$. This is equivalent to $\overline{P}\not\preccurlyeq Q$ and also to $\overline{Q}\not\preccurlyeq P$. The maximal element above $\overline{P}$ is not $Q$, so $\overline{P}\preccurlyeq P$. Similarly $\overline{Q}\preccurlyeq Q$. By Lemma \ref{largest PF}, $Tr_Q(S)=\emptyset$. Theorem \ref{Characterisation} implies that if $Q\in I$, then $\overline{Q} \in I$.

Since $I$ is not self-dual, Lemma \ref{check PF self-dual} implies that $P\in I$ and $\overline{P}\notin I$. By Theorem \ref{Characterisation}, $I$ must satisfy a Frobenius triangle $(P,x,y) \in Tr_P(S)$. We have $P+x+y=F$, so $x,y<x+y=\overline{P}$. This means that $\overline{P}\not\preccurlyeq x,y$ and so $\overline{Q}\preccurlyeq x,y\preccurlyeq Q$. Since $I$ satisfies the Frobenius triangle $(P,x,y)$, we know that $x\in I$ and $\overline{y}\not\in I$. Therefore $Q\in I$ and so $\overline{Q}\in I$.  Since $y\preccurlyeq Q$, we see that $\overline{Q}\preccurlyeq\overline{y}$ and therefore $\overline{y}\in I$. This contradicts the fact that $I$ satisfies the Frobenius triangle $(P,x,y)$.
\end{proof}

\begin{lemma}
If $GPF(S)$ is connected and $Q-P\not\in M(S)$, then $P(S)=2$.
\end{lemma}
\begin{proof}
In this case $S$ is triangle-free, so Proposition \ref{none P-Q} implies that $P(S)=2$.
\end{proof}

\begin{lemma}\label{not below Q}
Suppose $Q-P\in M(S)$. If $a\in M(S)$ satisfies $a\not\preccurlyeq Q$, then $Q-P\preccurlyeq a$.
\end{lemma}
\begin{proof}
Since $a\not\preccurlyeq Q$ we know that $a\preccurlyeq P$, that is, $P-a\in S$. Now $Q-a\not\in S$, so either $Q-a \in M(S)$ or $Q-a\in \h(S)\setminus M(S)$.
\begin{itemize}[wide,labelindent=0pt] 
    \item Case 1: Suppose $Q-a\in M(S)$. Now $Q-(Q-a)=a\not\in S$, which means $Q-a\not\preccurlyeq Q$, and so $Q-a\preccurlyeq P$. We have $P+a-Q\in S$, which implies $Q-P\preccurlyeq a$.
    \item Case 2: Suppose $Q-a\in \h(S)\setminus M(S)$. This means that $F+a-Q\in S$. But then,
    \[
    F-(Q-P) =(F+a-Q)+(P-a)\in S.
    \]
    This contradicts the fact that $Q-P\in M(S)$.
\end{itemize}
\end{proof}

\begin{lemma}
Suppose $GPF(S)$ is connected and $Q-P\in M(S)$. If $F+P=2Q$ then $P(S)=3$.  If $F+P \neq 2Q$ then $P(S)=4$.
\end{lemma}
\begin{proof}
Since $GPF(S)$ is connected there are $2$ self-dual order ideals of $(\M(S),\preccurlyeq)$.  Proposition \ref{condition to Frobenius triangle} implies that $Tr_P(S)\neq\emptyset$ and $Tr_Q(S)=\emptyset$. Suppose $I$ is an order ideal for which $I\cup S$ is a numerical set associated to $S$ and $I$ is not self-dual. By Theorem \ref{Characterisation}, if $Q \in I$ then $\overline{Q}\in I$.  Since $I$ is not self-dual, Lemma \ref{check PF self-dual} implies that $P\in I$ and $\overline{P}\not\in I$. Therefore $I$ must satisfy a Frobenius triangle $(P,x,y)\in Tr_P(S)$. This means that $P,x\in I$ and $\overline{y}\not\in I$.

The minimal elements of $(\M(S),\preccurlyeq)$ are $\overline{P}$ and $\overline{Q}$, so we must have either $\overline{P} \preccurlyeq y$ or $\overline{Q}\preccurlyeq y$.  Lemma \ref{superlemma} implies that if $\overline{P}\preccurlyeq y$, then $P- P = 0 \in M(S)$, which is not true.  Therefore, $\overline{P}\not\preccurlyeq y$, so $\overline{Q}\preccurlyeq y$. Lemma \ref{superlemma} implies that $x\preccurlyeq Q-P$. Since $x\in I$, we see that $Q-P\in I$.
\begin{itemize}[wide,labelindent=0pt] 
    \item Case 1: Suppose $x=Q-P$. In this case $\overline{y}=P+x=Q$, so $Q\not\in I$. Let 
    \[
    I_1=\{a\in M(S)\mid Q-P\preccurlyeq a\}.
    \]
Since $Q-P\in I$ we know that $I_1\subseteq I$. On the other hand given $a\in I$, we know that $a\not\preccurlyeq Q$.  Lemma \ref{not below Q} implies that $Q-P\preccurlyeq a$, that is, $a\in I_1$. We conclude that $I=I_1$.
\item Case 2: Suppose $x\prec Q-P$. Corollary \ref{PFtriless} implies that $x\prec Q$, and so $Q\in I$. Since $Tr_Q(S)=\emptyset$, Theorem \ref{Characterisation} implies that $\overline{Q}\in I$. 

Since $P-P\not\in M(S)$ and $Q-P\in M(S)$, Lemma \ref{superlemma} implies that $\overline{P}\not\preccurlyeq x$ and $\overline{Q}\preccurlyeq x$.  If $\overline{Q}\prec x$, then Lemma \ref{triless}  implies $\overline{Q}\preccurlyeq \overline{y}$. However, this is impossible since $\overline{Q}\in I$ and $\overline{y}\not\in I$. Therefore $\overline{Q}=x$. We have $\overline{y}=P+x=P+F-Q=\overline{Q-P}$.

Let
\[
I_2=\{a\in M(S)\mid \overline{Q}\preccurlyeq a\}.
\]
Since $\overline{Q}\in I$ we see that $I_2\subseteq I$. On the other hand if $a\in M(S)\setminus I_2$, then $\overline{Q}\not\preccurlyeq a$. This means that $\overline{a}\not\preccurlyeq Q$. Lemma \ref{not below Q} implies that $Q-P\preccurlyeq\overline{a}$ and so $a\preccurlyeq \overline{Q-P}$. Since $\overline{Q-P}=\overline{y}\not\in I$, we see that $a\not\in I$. We conclude that $I=I_2$.
\end{itemize}

We have seen that if $I\cup S$ is a numerical set associated to $S$ then $I \in \{\emptyset, M(S), I_1,I_2\}$. We see that $I_1=I_2$ if and only if $\overline{Q}=Q-P$, or equivalently, $F+P=2Q$. Therefore  $F+P=2Q$ implies $P(S)=3$, and $F+P\neq 2Q$ implies $P(S)=4$.
\end{proof}
\noindent Note that $S^*=S\cup M(S)$ and $(S\cup I_1)^*= S\cup I_2$.
We summarize the results of this section.
\begin{theorem}\label{type 3 characterize}
Let $S$ be a numerical semigroup of type $3$. Suppose $PF(S)=\{P,Q,F\}$ with $P<Q<F$.
\begin{itemize}
    \item If $P+Q-F\not\in S$, then $P(S)=4$.
    \item If $P+Q-F\in S$ and $Q-P\not\in M(S)$, then $P(S)=2$.
    \item If $P+Q-F\in S$, $Q-P\in M(S)$ and $F+P=2Q$, then $P(S)=3$.
    \item If $P+Q-F\in S$, $Q-P\in M(S)$ and $F+P\neq 2Q$, then $P(S)=4$.
\end{itemize}
\end{theorem}

\begin{corollary}
Let $S$ be a numerical semigroup of type $3$. Suppose $PF(S)=\{P,Q,F\}$ with $P<Q<F$. Then $S$ is P-minimal if and only if $Q-P\not\in M(S)$.
\end{corollary}
\begin{proof}
If $Q-P\notin M(S)$, then $S$ is triangle-free and Proposition \ref{none P-Q} implies that $S$ is P-minimal.

Conversely, suppose $S$ is P-minimal.  By Theorem \ref{type 3 characterize} either $GPF(S)$ is connected and $Q-P\notin M(S)$ or $GPF(S)$ has two connected components. In the first case we have nothing to prove. 

Suppose $GPF(S)$ has two connected components, so $P(S) = 4$.  Assume for the sake of contradiction that $Q-P\in M(S)$. Then $Q-(Q-P)=P\notin S$, that is, $Q-P\not\preccurlyeq Q$, and so $Q-P\preccurlyeq P$. Since $GPF(S)$ is not connected $P+Q-F\notin S$, which means $\overline{Q}\not\preccurlyeq P$. This implies $\overline{Q}\not\preccurlyeq Q-P$, and so $\overline{P}\preccurlyeq Q-P$. However this is impossible since $(Q-P)-\overline{P}=Q-F<0$. Therefore $Q-P\notin M(S)$.
\end{proof}

\begin{example}
\begin{enumerate}[wide, labelindent=0pt]  
\item Consider $S_1=\langle10, 19, 21, 36, 47\rangle$, which has $PF(S_1)=\{37,53,64\}$. Since $37+53-64\notin S_1$, this belongs to the first case of Theorem \ref{type 3 characterize} and $P(S_1) = 4$. We see that $GPF(S_1)$ is not connected.

\item Consider $S_2=\langle8, 9, 15, 21, 28\rangle$, which has $PF(S_2)=\{19,20,22\}$. Since $19+20-22\in S_2$ and $20-19\notin M(S_2)$, this belongs to the second case of Theorem \ref{type 3 characterize} and  $P(S_2) = 2$.  We see that $GPF(S_2)$ is connected and $Tr(S_2)=\emptyset$.

\item Consider $S_3=\langle8, 13, 22, 27\rangle$, which has $PF(S_3)=\{31,36,41\}$. Since $31+36-41\in S_3,\ 36-31\in M(S_3)$ and $41+31=2\cdot 36$, it belongs to the third case of Theorem \ref{type 3 characterize} and $P(S_3) = 3$.

\item Consider $S_4=\langle25, 29, 32, 45\rangle$, which has $PF(S_4)=\{71,142,155\}$. Since $71+142-155\in S_4,\ 142-71\in M(S_4)$ and $71+155\neq 2\cdot 142$, it belongs to the fourth case of Theorem \ref{type 3 characterize} and $P(S_4)=4$.
\end{enumerate}

Figure \ref{fig t3} shows the void poset and the GPF graph for $S_1,S_2,S_3$ and $S_4$.
\begin{figure}[!tbp]
  \begin{subfigure}[b]{0.2\textwidth}
     \begin{tikzpicture}
     \node (1) at (1,0) {$53$};
     \node (2) at (0,-1) {$37$};
     \node (3) at (1,-1) {$32$};
     \node (4) at (0,-2) {$27$};
     \node (5) at (1,-2) {$11$};
     \draw [black] (1) -- (3);
     \draw [black] (5) -- (3);
     \draw [black] (2) -- (4);
     \end{tikzpicture}
    \caption{$(\M(S_1),\preccurlyeq)$}
  \end{subfigure}
  \begin{subfigure}[b]{0.2\textwidth}
    \begin{tikzpicture}[main/.style = {draw, circle}] 
    \node[main] (1) at (0,0) {$53$};
    \node[main] (2) at (0,-1.4) {$37$};
    \draw (1) to [out=45,in=135,looseness=4] (1);
    \draw (2) to [out=45,in=135,looseness=4] (2);
    \end{tikzpicture}
    \caption{$GPF(S_1)$}
  \end{subfigure}
  \begin{subfigure}[b]{0.25\textwidth}
     \begin{tikzpicture}
     \node (1) at (0,0) {$19$};
     \node (2) at (1,0) {$20$};
     \node (3) at (0,-1) {$10$};
     \node (4) at (1,-1) {$11$};
     \node (5) at (2,-1) {$12$};
     \node (6) at (1,-2) {$2$};
     \node (7) at (2,-2) {$3$};
     \draw [black] (1) -- (3);
     \draw [black] (1) -- (4);
     \draw [black] (2) -- (4);
     \draw [black] (2) -- (5);
     \draw [black] (3) -- (6);
     \draw [black] (4) -- (6);
     \draw [black] (4) -- (7);
     \draw [black] (5) -- (7);
     \end{tikzpicture}
    \caption{$(\M(S_2),\preccurlyeq)$}
  \end{subfigure}
   \begin{subfigure}[b]{0.2\textwidth}
    \begin{tikzpicture}[main/.style = {draw, circle}] 
    \node[main] (1) at (0,0) {$20$};
    \node[main] (2) at (0,-1.4) {$19$};
    \draw (1) to [out=45,in=135,looseness=4] (1);
    \draw (1) -- (2);
    \end{tikzpicture}
    \caption{$GPF(S_2)$}
  \end{subfigure}
  \begin{subfigure}[b]{0.2\textwidth}
     \begin{tikzpicture}
     \node (1) at (0,0) {$36$};
     \node (2) at (1,0) {$31$};
     \node (3) at (0,-1) {$23$};
     \node (4) at (1,-1) {$18$};
     \node (5) at (0,-2) {$10$};
     \node (6) at (1,-2) {$5$};
     \draw [black] (1) -- (3);
     \draw [black] (5) -- (3);
     \draw [black] (2) -- (4);
     \draw [black] (6) -- (4);
     \draw [black] (2) -- (3);
     \draw [black] (5) -- (4);
     \end{tikzpicture}
    \caption{$(\M(S_3),\preccurlyeq)$}
  \end{subfigure}
  \begin{subfigure}[b]{0.2\textwidth}
    \begin{tikzpicture}[main/.style = {draw, circle}] 
    \node[main] (1) at (0,0) {$31$};
    \node[main] (2) at (0,-1.4) {$36$};
    \draw (1) to [out=45,in=135,looseness=4] (1);
    \draw (1) -- (2);
    \end{tikzpicture}
    \caption{$GPF(S_3)$}
  \end{subfigure}
  \begin{subfigure}[b]{0.35\textwidth}
     \begin{tikzpicture}
     \node (1) at (0,-0.2) {$142$};
     \node (2) at (-1,-0.8) {$113$};
     \node (3) at (1,-0.8) {$117$};
     \node (4) at (-1,-3.8) {$84$};
     \node (5) at (0,-1.6) {$88$};
     \node (6) at (1,-1.6) {$92$};
     \node (7) at (0,-2.4) {$63$};
     \node (8) at (1,-2.4) {$67$};
     \node (9) at (2,-0.2) {$71$};
     \node (10) at (0,-3.2) {$38$};
     \node (11) at (2,-3.2) {$42$};
     \node (12) at (1,-3.8) {$13$};
     \draw [black] (1) -- (2);
     \draw [black] (1) -- (3);
     \draw [black] (2) -- (4);
     \draw [black] (2) -- (5);
     \draw [black] (3) -- (5);
     \draw [black] (3) -- (6);
     \draw [black] (5) -- (7);
     \draw [black] (6) -- (7);
     \draw [black] (6) -- (8);
     \draw [black] (10) -- (7);
     \draw [black] (10) -- (8);
     \draw [black] (11) -- (8);
     \draw [black] (11) -- (9);
     \draw [black] (12) -- (10);
     \draw [black] (12) -- (11);
     \end{tikzpicture}
    \caption{$(\M(S_4),\preccurlyeq)$}
  \end{subfigure}
  \begin{subfigure}[b]{0.2\textwidth}
    \begin{tikzpicture}[main/.style = {draw, circle}] 
    \node[main] (1) at (0,0) {$142$};
    \node[main] (2) at (0,-1.4) {$71$};
    \draw (1) to [out=45,in=135,looseness=4] (1);
    \draw (1) -- (2);
    \end{tikzpicture}
    \caption{$GPF(S_4)$}
  \end{subfigure}
  \caption{Void poset and $GPF$ graph of $S_1=\langle10, 19, 21, 36, 47\rangle$, $S_2=\langle8, 9, 15, 21, 28\rangle$, $S_3=\langle8, 13, 22, 27\rangle$ and $S_4=\langle25, 29, 32, 45\rangle$}
   \label{fig t3}
\end{figure}
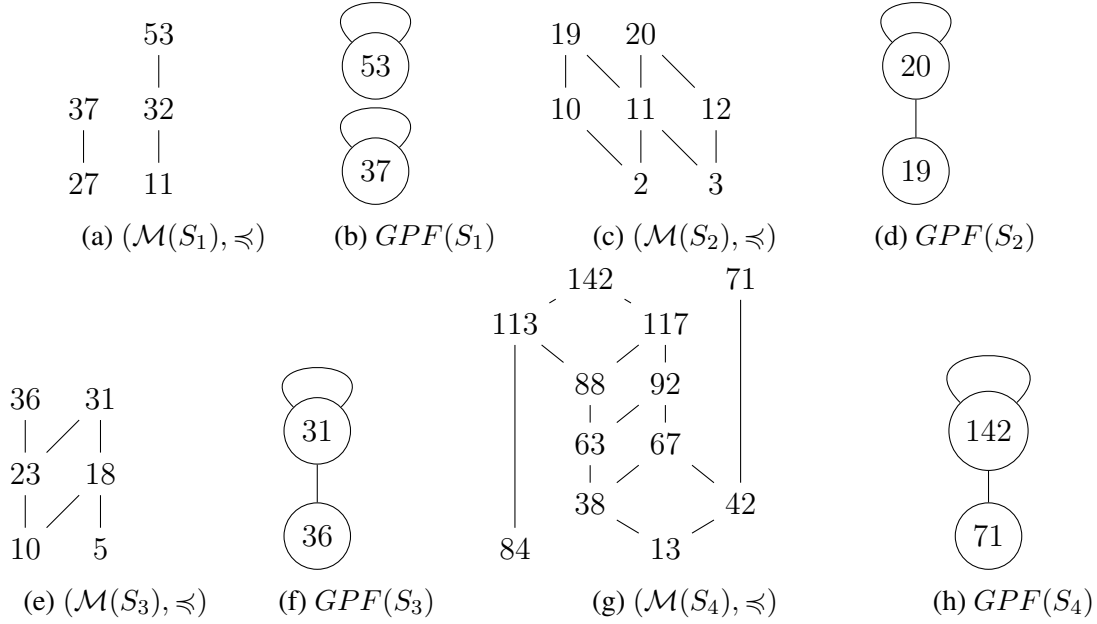
\end{example}

\section{Numerical semigroups of type 4}

In this section we show that there is a significant change in the behavior of $P(S)$ as we move from numerical semigroups of type $3$ to numerical semigroups of type $4$.  We show that there exist numerical semigroups $S$ with $t(S) = 4$ and $P(S)$ arbitrarily large.

\begin{proposition}\label{t=4 P(S) large}
Let
\[
S_n=\{0,5,10,15,\dots,5n\rightarrow\}.
\]
We have $t(S_n)=4$ and $P(S_n)=2n+4$.
\end{proposition}
\begin{proof}
We see that $F(S_n)=5n-1$ and 
\[
M(S_n)=\{1,6,\dots, 5n-4\}\cup\{2,7,\dots, 5n-3\}\cup\{3,8,\dots,5n-2\}.
\]
The void poset $(\M(S_n),\preccurlyeq)$ is the union of $3$ chains $1\preccurlyeq 6\preccurlyeq\dots\preccurlyeq 5n-4$, $2\preccurlyeq 7\preccurlyeq\dots\preccurlyeq 5n-3$ and $3\preccurlyeq 8\preccurlyeq\dots\preccurlyeq 5n-2$. See Figure \ref{fig:void poset S6} for a depiction of $(\M(S_6),\preccurlyeq)$.
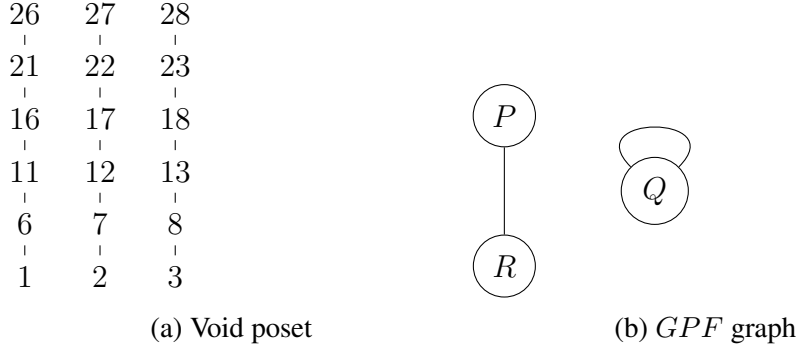
\begin{figure}[!tbp]
  \begin{subfigure}[b]{0.4\textwidth}
     \begin{tikzpicture}
     \node (1) at (0,0) {$26$};
     \node (2) at (1,0) {$27$};
     \node (3) at (2,0) {$28$};
     \node (4) at (0,-0.7) {$21$};
     \node (5) at (1,-0.7) {$22$};
     \node (6) at (2,-0.7) {$23$};
     \node (7) at (0,-1.4) {$16$};
     \node (8) at (1,-1.4) {$17$};
     \node (9) at (2,-1.4) {$18$};
     \node (10) at (0,-2.1) {$11$};
     \node (11) at (1,-2.1) {$12$};
     \node (12) at (2,-2.1) {$13$};
     \node (13) at (0,-2.8) {$6$};
     \node (14) at (1,-2.8) {$7$};
     \node (15) at (2,-2.8) {$8$};
     \node (16) at (0,-3.5) {$1$};
     \node (17) at (1,-3.5) {$2$};
     \node (18) at (2,-3.5) {$3$};
     \draw [black] (1) -- (4);
     \draw [black] (2) -- (5);
     \draw [black] (3) -- (6);
     \draw [black] (7) -- (4);
     \draw [black] (8) -- (5);
     \draw [black] (9) -- (6);
     \draw [black] (7) -- (10);
     \draw [black] (8) -- (11);
     \draw [black] (9) -- (12);
     \draw [black] (13) -- (10);
     \draw [black] (14) -- (11);
     \draw [black] (15) -- (12);
     \draw [black] (13) -- (16);
     \draw [black] (14) -- (17);
     \draw [black] (15) -- (18);
     \end{tikzpicture}
    \caption{Void poset}
  \end{subfigure}
  \begin{subfigure}[b]{0.4\textwidth}
    \begin{tikzpicture}[main/.style = {draw, circle}] 
    \node[main] (1) at (0,0) {$P$};
    \node[main] (2) at (0,-2) {$R$};
    \node[main] (3) at (2,-1) {$Q$};
    \draw (3) to [out=45,in=135,looseness=4] (3);
    \draw (1) -- (2);
    \end{tikzpicture}
    \caption{$GPF$ graph}
  \end{subfigure}
  \caption{Void poset and $GPF$ graph of $S_6=\langle5,31,32,33,34\rangle$.}
    \label{fig:void poset S6}
\end{figure}

Since $x \in PF(S_n)$ implies that $5+x \in S$, it is easy to check that $PF(S_n)=\{5n-4,5n-3,5n-2,5n-1\}$, and so $t(S_n) = 4$.  Let $P=5n-4,\ Q=5n-3,\ R=5n-2$, and $F = F(S_n) = 5n-1$. It is easy to check that $GPF(S_n)$ has an edge between $P$ and $R$ and a loop on $Q$. So, $GPF(S_n)$ has two connected components, and so there are $4$ self-dual order ideals of $(\M(S_n),\preccurlyeq)$.

Lemma \ref{largest PF} implies that $Tr_R(S_n)=\emptyset$. From the description of $M(S_n)$, we can check that
\[
Tr(S_n) = \{(5n-3,1,1),\ (5n-4,1,2),\ (5n-4,2,1)\}.
\]
Suppose $I$ is an order ideal of $(\M(S_n),\preccurlyeq)$ that is not self-dual and $I\cup S$ is a numerical set associated to $S$. Lemma \ref{check PF self-dual} implies that either $P\in I$ and $\overline{P}\notin I$, or $Q\in I$ and $\overline{Q}\not\in I$.
\begin{itemize}[wide,labelindent=0pt] 
    \item Case $1$: Suppose $Q=5n-3\in I$ and $\overline{Q}=2\not\in I$.  Theorem \ref{Characterisation} implies that $I$ has to satisfy a Frobenius triangle in $Tr_Q(S_n)$, and therefore $I$ must satisfy $(5n-3,1,1)$. This means $1\in I$ and $\overline{1} = 5n-2\not\in I$. Since $1\in I$, we see that $\{1,6,\dots,5n-4\}\subseteq I$. Since $5n-2\not\in I$, we see that $\{3,8,\dots,5n-2\}\cap I=\emptyset$.
    
Since $P \in I \cap PF(S_n)$ and $\overline{P} \not\in I$, Theorem \ref{Characterisation} implies that $I$ satisfies a Frobenius triangle in $Tr_P(S_n)$.  We see that $I$ does not satisfy $(5n-4,1,2)$ because $\overline{2}=5n-3\in I$. We see that $I$ does not satisfy $(5n-4,2,1)$ because $2\not\in I$.  This is a contradiction.

    \item Case $2$: Suppose $P=5n-4\in I$ and $\overline{P}=3\not\in I$. Theorem \ref{Characterisation} implies that $I$ satisfies a Frobenius triangle in $Tr_P(S_n)$.  First consider the case when $I$ satisfies $(5n-4,1,2)$. This means $1\in I$ and $\overline{2}=5n-3\not\in I$. This implies $\{1,6,\dots,5n-4\}\subseteq I$ and $\{2,7,\dots,5n-3\}\cap I=\emptyset$. The remaining elements of $M(S_n)$ are $\{8,11,\dots,5n-2\}$. Once we decide the first element in this list to include in $I$ then all larger elements must also be in $I$. This gives $n$ choices including the choice to not include any elements from this list. Since $\overline{5n-2}=1\in I$, Theorem \ref{Characterisation} implies that each of these choices leads to a numerical set associated to~$S$.   
      
    Next consider the case when $I$ satisfies $(5n-4,2,1)$. This means $2\in I$ and $\overline{1}=5n-2\not\in I$. This implies $\{2,7,\dots,5n-3\}\subseteq I$ and $\{3,8,\dots,5n-2\}\cap I=\emptyset$.  Note that $\overline{Q} = 2 \in I$.  The remaining elements of $M(S_n)$ are $\{1,6,\dots,5n-9\}$. As above, once we decide the first element in this list to include in $I$ then all larger elements must also be in $I$.  Each of these $n$ choices leads to a numerical set associated to $S$.   
    \end{itemize}
We have seen that $(\M(S_n),\preccurlyeq)$ has $4$ self-dual order ideals and $2n$ order ideals $I$ that are not self dual for which $I \cup S$ is a numerical set associated to $S$. Therefore, $P(S_n)=2n+4$.
\end{proof}

\section{Numerical Semigroups of Large Type}\label{sec:large}

In this section we focus on two families of numerical semigroups of large type.  
\subsection{Numerical semigroups with $P(S) = 2$ and large type} We first describe a family of numerical semigroups where every member has $P(S) = 2$ that includes semigroups of arbitrarily large type.

\begin{proposition}\label{t large P(S)=2}
Let $n\geq1,\ m = 2n+1$, and 
\[
S_n=\{0,2m\rightarrow\}\cup\{m+2k \mid 0 \leq k \leq n-1\}.
\]
Then $S_n$ is a numerical semigroup with $t(S_n)=n+1$ and $P(S_n) = 2$.
\end{proposition}
The family we consider is a subset of the family in \cite[Proposition 16]{NathanPartition}. They show that each member of the family has $P(S) = 2$ and $|M(S)|$ can be arbitrarily large.  They do not discuss the psedo-Frobenius numbers or the type of these semigroups. 

\begin{proof}
Note that $F=F(S_n)=2m-1=4n+1$. We see that $S_n$ is a numerical semigroup since all nonzero elements of $S_n$ are larger that $\frac{F}{2}$.
The void of $S_n$ is
\[
M(S_n)=\{2i+1\mid 0\leq i\leq n-1\}\cup \{2i\mid n+1\leq i\leq 2n \}.
\]
For $k\in [0,n-1],\ 2k+1,\ 2(k+n+1)\in M(S_n)$ and
\[
2(k+n+1)-(2k+1)=2n+1=m\in S_n.
\]
This means that $2k+1\preccurlyeq 2(k+n+1)$ in $(\M(S_n),\preccurlyeq)$, and so $2k+1\notin PF(S_n)$. On the other hand, for $k\in [n+1,2n]$,
\[
2k+m\geq 2(n+1)+2n+1=4n+3>F.
\]
Since $m$ is the smallest nonzero element of $S_n$, we see that $2k+S_n\setminus\{0\}\subseteq S_n$, and so $2k\in PF(S_n)$. Therefore,
\[
PF(S_n)\setminus\{F\}=\{2k\mid n+1\leq k\leq 2n\},
\]
which means $t(S_n)=n+1$. 

A main idea for the rest of the proof is to show that each $S_n$ is triangle-free and that $GPF(S_n)$ is connected.  Proposition \ref{none P-Q} then implies that $P(S_n) = 2$.  See Figure \ref{fig:void poset S3} for the poset $(\M(S_3),\preccurlyeq)$ and $GPF(S_3)$.

\begin{figure}[!tbp]
  \begin{subfigure}[b]{0.4\textwidth}
     \begin{tikzpicture}
     \node (1) at (0,0) {$8$};
     \node (2) at (1,0) {$10$};
     \node (3) at (2,0) {$12$};
     \node (4) at (0,-1) {$1$};
     \node (5) at (1,-1) {$3$};
     \node (6) at (2,-1) {$5$};
     \draw [black] (1) -- (4);
     \draw [black] (2) -- (5);
     \draw [black] (2) -- (4);
     \draw [black] (3) -- (4);
     \draw [black] (3) -- (5);
     \draw [black] (3) -- (6);
     \end{tikzpicture}
    \caption{Void poset}
  \end{subfigure}
  \begin{subfigure}[b]{0.4\textwidth}
    \begin{tikzpicture}[main/.style = {draw, circle}] 
    \node[main] (1) at (0,0) {$8$};
    \node[main] (2) at (2,0) {$12$};
    \node[main] (3) at (4,0) {$10$};
    \draw (1) -- (2);
    \draw (2) -- (3);
    \draw (2) to [out=45,in=135,looseness=5] (2);
    \end{tikzpicture} 
    \caption{$GPF$ graph}
  \end{subfigure}
  \caption{$(\M(S_3),\preccurlyeq)$ and $GPF(S_3)$ where $S_3=\{0,7,9,11,14\rightarrow\}$.}
    \label{fig:void poset S3}
\end{figure}
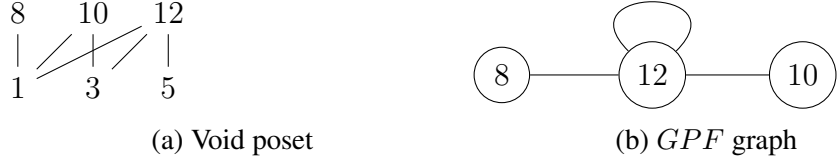

Let $P,Q \in PF(S) \setminus \{F\}$ with $P<Q$.  Since $P$ and $Q$ are both even so is $P-Q$, and $0<P-Q\leq 2n$.  Therefore, $P-Q\notin M(S_n)$. Proposition \ref{condition to Frobenius triangle} implies that $Tr(S_n)=\emptyset$, or equivalently, that $S_n$ is triangle-free. By Proposition \ref{none P-Q}, $P(S_n)=2^{\kappa}$ where $\kappa$ is the number of connected components of $GPF(S_n)$.  To complete the proof, we need only show that $GPF(S_n)$ is connected.

Suppose that $k \in [n+1,2n]$.  We have that $2k,\ 4n\in PF(S_n)$ and
\[
2k+4n-F=2k-1=2n+1+2(k-n-1)\in S_n.
\]
We see that there is an edge between $2k$ and $4n$ in $GPF(S_n)$. We conclude that $GPF(S_n)$ is connected, completing the proof.
\end{proof}

Propositions \ref{t=4 P(S) large} and \ref{t large P(S)=2}  make it clear that for semigroups of type at least $4$, the connection between the set of pseudo-Frobenius numbers $PF(S)$ and $P(S)$ is not so clear.

\subsection{Maximal Embedding Dimension Numerical Semigroups}

Recall that $S$ has maximal embedding dimension if and only if its number of minimal generators is equal to its multiplicity, that is, $e(S) = m(S)$.  We characterize $P(S)$ for numerical semigroups $S$ that have maximal embedding dimension and are triangle-free.
\begin{theorem}
Suppose $S$ is a triangle-free numerical semigroup with maximal embedding dimension. Let $m=m(S)$ and $F=F(S)$.
\begin{enumerate}[wide, labelindent=0pt]  
    \item Suppose $x,y \in M(S)$ with $x \le y$.  Then $x\preccurlyeq y$ in $(\M(S),\preccurlyeq)$ if and only if $m \mid (y-x)$.
    \item If $m$ is odd then $P(S)=2^{\frac{m-1}{2}}$.
    \item If $m$ is even and $F$ is odd then $P(S)=2^{\frac{m-2}{2}}$.
    \item If $m$ and $F$ are both even then $P(S)=2^{\frac{m}{2}}$.
\end{enumerate}
\end{theorem}
\begin{proof}
Since $S$ is of maximal embedding dimension, $|PF(S)|=m-1$. For distinct $P,Q\in PF(S)$, we know that $P-Q\notin S$. This implies that $P\not\equiv Q\pmod{m}$. Moreover since pseudo-Frobenius numbers are gaps, none of them is $0$ modulo $m$.  Therefore we can label $PF(S)=\{P_1,\dots,P_{m-1}\}$ with $P_i\equiv i\pmod{m}$.

If $x,y\in M(S)$ satisfy $x\le y$ and $m \mid (y-x)$, then clearly $x\preccurlyeq y$. Assume for the sake of contradiction that there are $x,y\in M(S)$ satisfying $x\le y,\ m\nmid (y-x)$, and  $x\preccurlyeq y$. Note that none of $y,\ F-x$, and $y-x$ is divisible by $m$.  Suppose that $P_i\equiv y \pmod{m},\ P_j\equiv F-x \pmod{m}$, and $P_k\equiv y-x\pmod{m}$. Since $y, F-x\notin S$, we know that $y\leq P_i$ and $F-x\leq P_j$. Since $y-x\in S$, we know that $y-x>P_k$. Now $P_j\equiv F-(P_i-P_k)\pmod{m}$ and
\[
F-(P_i-P_k)<(P_j+x)-y+(y-x)=P_j,
\]
so $F-(P_i-P_k)\notin S$. Since $P_i, P_j\in PF(S)$, we have $P_i-P_j\notin S$. This means that $P_i-P_j\in M(S)$, but this contradicts the fact that $S$ is triangle-free. This completes the proof of (1).

Since $S$ is triangle-free, Proposition \ref{none P-Q} implies that $P(S)=2^{\kappa(S)}$, where $\kappa(S)$ is the number of connected components of $GPF(S)$.  Suppose $P_i,P_j\in PF(S)\setminus\{F\}$. We see that $F-P_i\preccurlyeq P_j$ if and only if $P_i+P_j\equiv F\pmod{m}$.
This means that the graph $GPF(S)$ mostly consists of components of size $2$ except when $2P_i\equiv F \pmod{m}$, in which case $P_i$ is its own component has a loop on it. The description of $PF(S)$ given above shows that $GPF(S)$ has $m-2$ vertices. 
\begin{itemize}[wide, labelindent=0pt]  
    \item If $m$ is odd then there is exactly one $i$ for which $2P_i\equiv F \pmod{m}$. Therefore, $\kappa(S)=1+\frac{m-3}{2}$.
    \item If $m$ is even and $F$ is odd then there is no $i$ for which $2P_i\equiv F \pmod{m}$. Therefore, $\kappa(S) = \frac{m-2}{2}$.
    \item If $m$ and $F$ are both even then there are two $i$ for which $2P_i\equiv F \pmod{m}$.  Therefore, $\kappa(S) = 2+ \frac{m-4}{2}$.
\end{itemize}
\end{proof}

\begin{example}
For $m\ge 2$, let $S_m$ be the numerical semigroup generated by
\[
\{m\}\cup \{m(m+k-1)+k\mid 1\leq k\leq m-2\}\cup \{m(2m-1)+m-1\}.
\]
We can check that $F(S_m)=m(2m-2)+m-1$ and 
\[
PF(S_m)=\{m(m+k-2)+k\mid 1\leq k\leq m-2\}\cup \{m(2m-2)+m-1\}.
\]
We see that $S_m$ has maximal embedding dimension since $t(S_m)=m-1$. Next we check that $S_m$ is triangle-free. Suppose $P_j,P_i\in PF(S_m)\setminus\{F(S_m)\}$ with $P_j\equiv j\pmod{m}$, $P_i\equiv i\pmod{m}$ and $1\leq i<j\leq m-2$. Then
\[
P_j-P_i=m(m+j-2)+j-\left(m(m+i-2)+i\right)=m(j-i)+j-i.
\]
Let $k=m-1-(j-i)$.  We have
\[
F(S_m)-(P_j-P_i)=m(m+k-1)+k\in S_m.
\]
This means that $P_j-P_i\notin M(S_m)$ and so $S_m$ is triangle-free. When $m$ is odd we have $P(S_m)=2^{\frac{m-1}{2}}$.  When $m$ is even, we see that $F(S_m)$ is odd and $P(S_m)=2^{\frac{m-2}{2}}$.
\end{example}

When $S$ has maximal embedding dimension and is triangle-free, $(\M(S),\preccurlyeq)$ is a union of $m-2$ chains, one for each nonzero residue class modulo $m$ except the one containing $F(S)$.  When $S$ has maximal embedding dimension and is not triangle-free, for example when 
\[
S = \langle 15, 34, 38, 57, 61, 80, 84, 103, 107, 126, 130, 149, 153, 172, 176\rangle,
\] the structure of $(\M(S),\preccurlyeq)$ can be much more complicated.

\section{Acknowledgments}
This paper is based on research that started at the San Diego State University REU 2019.  It was supported by NSF-REU award 1851542. The second author received support from NSF grants DMS 1802281 and DMS 2154223.


\begin{thebibliography}{}


\bibitem{AdjointPaper}
E. Antokoletz and A. Miller, \textit{Symmetry and factorization of numerical sets and monoids}. J. Algebra 247 (2002), no. 2, 636--671.



\bibitem{Grobenius}
J. Arroyo, J. Autry, C. Crandall, J. Lefler, and V. Ponomarenko, \textit{On numerical semigroups with almost-maximal genus}. PUMP J. Undergrad. Res. 3 (2020), 62--67.

\bibitem{NathanPartition}
H. Constantin, B. Houston-Edwards, and N. Kaplan, \textit{Numerical sets, core partitions, and integer points in polytopes}. Combinatorial and Additive Number Theory. II, 99--127, Springer Proc. Math. Stat., 220, Springer, Cham, 2017.


\bibitem{numericalsgpsgap}
M.~Delgado, P.~Garc\'ia-S\'anchez, and J.~Morais, 
\textit{NumericalSgps, A package for numerical semigroups}, 
Version 1.1.10 (2018), (Refereed GAP package),
\url{https://gap-packages.github.io/numericalsgps/}.


\bibitem{Type}
R. Fr\"oberg, C. Gottlieb, and R. H\"aggkvist (1986).
\textit{On numerical semigroups}.
Semigroup Forum 35 (1987), no. 1, 63--83. 

\bibitem{sanchesTextbook}
P. A. Garc\'ia S\'anchez and J. C. Rosales, {Numerical  {S}emigroups}, Developments in Mathematics, 20. Springer, New York, 2009. x+181 pp.
vol 20, Springer.


\bibitem{Complementary numerical sets}
M. Guhl, J. Juarez, V. Ponomarenko, R. Rechkin, and D. Singhal, \textit{Complementary numerical sets}. Integers 22 (2022), Paper No. A17, 13 pp.


\bibitem{Almost Sym Arf}
N. G\"um\"u\c{s}ba\c{s}, N. Tuta\c{s}, and N. Er,
\textit{Almost symmetric {A}rf partitions}. 
Turkish J. Math. 44 (2020), no. 6, 2185--2198.


\bibitem{decomposing partitions}
H.I. Karaka\c{s} and N. Tuta\c{s}, \textit{A decomposition of partitions and numerical sets}. Semigroup Forum 101 (2020), no. 3, 704--715.


\bibitem{partition bijection}
W. Keith and R. Nath, \textit{Partitions with prescribed hooksets}. J. Comb. Number Theory 3 (2011), no. 1, 39--50.


\bibitem{Miller}
J. Marzuola and A. Miller, \textit{Counting numerical sets with no small atoms}. J. Combin. Theory Ser. A 117 (2010), no. 6, 650--667.

\bibitem{H Nari}
H. Nari, \textit{Symmetries on almost symmetric numerical semigroups}. Semigroup Forum 86 (2013), no. 1, 140--154.


\bibitem{With Lin}
D. Singhal and Y. Lin, \textit{Density of numerical sets associated to a numerical semigroup}. Comm. Algebra 49 (2021), no. 10, 4291--4303.


\bibitem{NS with given void}
A.M. Robles-P\'erez and J.C. Rosales, \textit{On the enumeration of the set of numerical semigroups with fixed Frobenius number and fixed number of second kind gaps}. Results Math. 77 (2022), no. 1, Paper No. 10, 18 pp.


\bibitem{Arf}
N. Tuta\c{s}, H.I. Karaka\c{s}, and N. G\"um\"u\c{s}ba\c{s} (2019).
\textit{Young tableaux and Arf partitions}.
Turkish J. Math. 43 (2019), no. 1, 448--459.

\end{thebibliography}
\end{document}